\documentclass[a4paper,11pt]{article}
\usepackage{amsmath, amsthm, amssymb}
\def\R{{\mathbb R}}

\def\P{{\mathbb P}}

\def\E{{\mathbb E}}

\newtheorem{thm}{Theorem}[section]

\newtheorem{lem}[thm]{Lemma}

\theoremstyle{remark}
\newtheorem{rem}[thm]{Remark}

\theoremstyle{definition}

\numberwithin{equation}{section}

\title{Johnson-Lindenstrauss lemma for circulant matrices}

\author{Aicke Hinrichs \\
Mathematisches Institut, Universit\"at Jena\\
Ernst-Abbe-Platz 2, 07740 Jena, Germany\\
email:\ a.hinrichs@uni-jena.de\\
\qquad
\\
Jan Vyb\'\i ral \\
Radon Institute for Computational and Applied Mathematics (RICAM)\\
Austrian Academy of Sciences\\
Altenbergerstra\ss e 69\\
A-4040 Linz, Austria\\
email:\ jan.vybiral@oeaw.ac.at}

\begin{document}
\maketitle

\begin{abstract}
We prove a variant of a Johnson-Lindenstrauss lemma for matrices with circulant
structure. This approach allows to minimise the randomness used, is easy to implement and provides
good running times. The price to be paid is the higher dimension of the target space
$k=O(\varepsilon^{-2}\log^3n)$ instead of the classical bound $k=O(\varepsilon^{-2}\log n)$.
\end{abstract}



\section{Introduction}

The classical Johnson-Lindenstrauss lemma may be formulated as follows.
\begin{thm}
Let $\varepsilon\in(0,\frac 12)$ and let $x_1,\dots,x_n\in \R^d$ be arbitrary points.
Let $k=O(\varepsilon^{-2}\log n)$ be a natural number. Then there exists a (linear) mapping
$f:\R^d\to \R^k$ such that
$$
 (1-\varepsilon)||x_i-x_j||_2^2\le ||f(x_i)-f(x_j)||_2^2\le (1+\varepsilon)||x_i-x_j||_2^2
$$
for all $i,j\in\{1,\dots,n\}.$ Here $||\cdot||_2$ stands for the Euclidean norm in $\R^d$ or $\R^k$, respectively.
\end{thm}

The original proof of Johnson and Lindenstrauss \cite{JL} uses (up to a scaling factor) an orthogonal projection
onto a random $k$-dimensional subspace of $\R^d$. We refer also to \cite{DG} for a beautiful and self-contained
proof. Later on, this lemma found many applications, especially in design of algorithms, where it sometimes 
allows to reduce the dimension of the underlying problem essentially and break the so-called ``curse of dimension'',
cf. \cite{IM} or \cite{IN}.

The evaluation of $f(x)$, where $f$ is a projection onto a random $k$ dimensional subspace, 
is a very time-consuming operation. Therefore, a significant effort was devoted to
\begin{itemize}
\item minimize the running time of $f(x)$,
\item minimize the memory used,
\item minimize the number of random bits used,
\item simplify the algorithm to allow an easy implementation.
\end{itemize}

Achlioptas observed in \cite{A}, that the mapping may also be realised by a matrix, where each component
is selected independently at random with a fixed distribution. This decreases  the time for
evaluation of $f(x)$ essentially.

An important breakthrough was achieved by Ailon and Chazelle in \cite{AC}. Let us briefly describe their
{\it Fast Johnson-Lindenstrauss transform} (FJLT). The FJLT is the product of three matrices
$f(x)=PHDx$, where
\begin{itemize}
\item $P$ is a $k\times d$ matrix, where each component is generated independently at random.
In particular, $P_{i,j}\approx N(0,1)$ with probability
$$
q=\min\left\{\Theta\left(\frac{\log^2n}{d}\right),1\right\}
$$
and $P_{i,j}=0$ with probability $1-q$,
\item $H$ is the $d\times d$ normalised Hadamard matrix,
\item $D$ is a random $d\times d$ diagonal matrix, with each $D_{i,i}$ drawn independently
from $\{-1,1\}$ with probability 1/2.
\end{itemize}

It follows, that with high probability, $f(x)$ may be calculated in time $O(d\log d+qd\varepsilon^{-2}\log n).$

We refer to \cite{M} for a historical overview as well as for an extensive description of the
present ``state of the art''.

In this note we propose another direction to approach the Johnson-Lindenstrauss lemma, namely
we investigate the possibility of taking a partial circulant matrix for $f$ combined with
a random $\pm 1$ diagonal matrix, see the next section for exact definitions.

This transform has a running time of $O(d\log d)$, requires less randomness ($2d$ instead of $kd$ or $(k+1)d$
used in \cite{A,AC2,AC}) and allows a simpler implementation.

Unfortunately, up to now, we were only able to prove the statement with $k=O(\varepsilon^{-2}\log^3 n)$,
compared to the standard value $k=O(\varepsilon^{-2}\log n)$. We leave
the possible improvements of this bound open for further investigations.

\section{Circulant matrices}

We study the question (which to our knowledge has not been addressed in the literature before),
whether $f$ in the Johnson-Lindenstrauss lemma may be chosen as a circulant matrix. 
Let us give the necessary notation.

Let $a=(a_0,\dots,a_{d-1})$ be independent identically distributed random variables. 
We denote by $M_{a,k}$ the partial circulant matrix
$$M_{a,k}=\left(
\begin{matrix}
a_0 & a_1 & a_2 & \dots & a_{d-1}\\
a_{d-1} & a_0 & a_1 & \dots & a_{d-2}\\
a_{d-2} & a_{d-1} & a_{0} &\dots & a_{d-3}\\
\vdots & \vdots & \vdots &\ddots & \vdots\\
a_{d-k+1} & a_{d-k+2} & a_{d-k+3} & \dots & a_{d-k}
\end{matrix}
\right).
$$

Furthermore, if $\varkappa=(\varkappa_0,\dots,\varkappa_{d-1})$ are independent Bernoulli
variables, we put
$$
 D_{\varkappa}=\left(\begin{matrix}
\varkappa_0 & 0  & \dots & 0\\
0 & \varkappa_1  & \dots & 0\\
\vdots & \vdots  &\ddots & \vdots\\
0 & 0 & \dots & \varkappa_{d-1}
\end{matrix}\right).
$$

\begin{thm}\label{thm1}
Let $x_1,\dots,x_n$ be arbitrary points in $\R^d$, let $\varepsilon\in (0,\frac 12)$
and let $k=O(\varepsilon^{-2}\log^3 n)$ be a natural number.
Let $a=(a_0,\dots,a_{d-1})$ be independent Bernoulli variables or independent
normally distributed variables. Let $M_{a,k}$ and $D_\varkappa$ 
be as above and put $f(x)=\frac{1}{\sqrt k}M_{a,k} D_{\varkappa}x$.

Then with probability at least 2/3 the following holds
$$
(1-\varepsilon)||x_i-x_j||_2^2\le ||f(x_i)-f(x_j)||_2^2\le (1+\varepsilon)||x_i-x_j||_2^2,\qquad i,j=1,\dots,n.
$$
\end{thm}

The preconditioning of $x$ using $D_{\varkappa}$ seems to be necessary and we 
shall comment on this point later on.
Its role may be compared with the use of the random Fourier transform in \cite{AC}.

In contrast to the above mentioned variants of the Johnson-Lindenstrauss lemma,
the coordinates of $f(x)$ are now no longer independent random variables.
Our approach ``decouples''
the dependence caused by the circulant structure. It resembles in some aspects the methods
used recently in compressed sensing, cf. \cite{B1,B2,R}.

First, we recall the Lemma 1 from Section 4.1 of \cite{LM}
(cf. also Lemma 2.2 of \cite{M}), which shall be useful later on.

\begin{lem}\label{LemLM} Let
$$
Z=\sum_{i=1}^D \alpha_i(a_i^2-1),
$$
where $a_i$ are i.i.d. normal variables and $\alpha_i$ 
are nonnegative real numbers. Then for any $t>0$
\begin{align*}
{\mathbb P}(Z\ge 2||\alpha||_2\sqrt{t}+2||\alpha||_\infty t)&\le \exp(-t),\\
{\mathbb P}(Z\le -2||\alpha||_2\sqrt{t})&\le \exp(-t).
\end{align*}
\end{lem}

Furthermore, we shall use the decoupling lemma of \cite[Proposition 1.9]{BT}.

\begin{lem}\label{decoup}
Let $\xi_0,\dots,\xi_{d-1}$ be independent random variables with 
${\mathbb E}\,\xi_0=\dots={\mathbb E}\,\xi_{d-1}=0$ and let $\{x_{i,j}\}_{i,j=0}^{d-1}$
be a double sequence of real numbers. Then for $1\le p <\infty$
$$
{\mathbb E}\biggl|\sum_{i\not=j}x_{i,j}\xi_i \xi_j\biggr|^p
\le 4^p {\mathbb E}\biggl|\sum_{i\not= j}x_{i,j}\xi_i\xi'_j\biggr|^p,
$$
where $(\xi'_0,\dots,\xi'_{d-1})$ denotes an independent copy of
$(\xi_0,\dots,\xi_{d-1})$.
\end{lem}

The key role in the proof of the Johnson-Lindenstrauss lemma is played by the following
estimates.

\begin{lem}\label{lem1}
Let $k\le d$ be natural numbers and let $\varepsilon \in (0, \frac12)$. 
Let $a=(a_0,\dots,a_{d-1})$, $M_{a,k}$ and $D_\varkappa$ 
be as in Theorem \ref{thm1} and let $x\in \R^d$ be a unit vector.
Put $f(x)=M_{a,k} D_{\varkappa}x$.

Then there is a constant $c$, independent on $k, d,\varepsilon$ and $x$, such that
$$
 {\mathbb P_{a,\varkappa}}\Bigl(||f(x)||_2^2\ge (1+\varepsilon)k\Bigr)\le \exp(-c(k\varepsilon^{2})^{1/3})
$$
and
$$
 {\mathbb P_{a,\varkappa}}\Bigl(||f(x)||_2^2\le (1-\varepsilon)k\Bigr)\le\exp(-c(k\varepsilon^{2})^{1/3}).
$$
\end{lem}
\begin{proof}
Let $S:\R^d \to \R^d$ denote the shift operator
$$
 S(x_0,x_1,\dots,x_{d-1})=(x_{d-1},x_0,x_1,\dots,x_{d-2}),\quad x\in\R^d.
$$
Then
$$
||f(x)||_2^2=||M_{a,k} D_{\varkappa} x||_2^2 = \sum_{j=0}^{k-1}|\langle S^j a,D_\varkappa x\rangle|^2
=\sum_{j=0}^{k-1}\left(\sum_{i=0}^{d-1}a_i\varkappa_{j+i}x_{j+i}\right)^2
=I + II,
$$
where
$$
I=\sum_{i=0}^{d-1}a_i^2 \cdot \sum_{j=0}^{k-1}x^2_{j+i}
$$
and
$$
II=\sum_{j=0}^{k-1}\sum_{i\not=i'} a_i a_{i'}\varkappa_{j+i}\varkappa_{j+i'}x_{j+i}x_{j+i'}.
$$
Here (and any time later) the summation in the index is to be understood modulo $d$.

The decoupling of the circulant matrix is based on
\begin{equation}\label{eq:fin1}
\P_{a,\varkappa}\Bigl(||M_{a,k}D_\varkappa x||_2^2\ge (1+\varepsilon)k\Bigr)\le
\P_{a}(I\ge (1+\varepsilon/2)k)+\P_{a,\varkappa}(II\ge \varepsilon k/2)
\end{equation}
and
\begin{equation}\label{eq:fin2}
\P_{a,\varkappa}\Bigl(||M_{a,k}D_\varkappa x||_2^2\le (1-\varepsilon)k\Bigr)\le
\P_{a}(I\le (1-\varepsilon/2)k)+\P_{a,\varkappa}(II\le -\varepsilon k/2).
\end{equation}

We use Lemma \ref{LemLM} to estimate the diagonal term $I$.

We choose $\alpha_i=\displaystyle\sum_{j=0}^{k-1}x^2_{j+i}$ 
and get $||\alpha||_1=k, ||\alpha||_\infty\le 1$
and hence $||\alpha||_2\le \sqrt k$.
This leads to
\begin{equation}\label{eq:11}
{\mathbb P_a}(I\le k-2\sqrt{kt})\le \exp(-t)
\end{equation}
and
\begin{equation}\label{eq:12}
{\mathbb P_a}(I\ge k+2\sqrt{kt}+2t)\le \exp(-t).
\end{equation}
We set $\varepsilon k/2=2\sqrt{kt}$, i.e. $t=\varepsilon^2k/16$, in \eqref{eq:11}
and obtain
\begin{equation}\label{eq:fin3}
{\mathbb P}_{a}(I\le (1-\varepsilon/2)k)\le \exp(-\varepsilon^2k/16).
\end{equation}
On the other hand, if $c=5/2-\sqrt6>1/20$, then $\sqrt c+c/2=1/4$ and
$$
2\sqrt{kt}+2t\le \varepsilon k/2
$$
for $t=c\varepsilon^2 k$, which finally gives
\begin{equation}\label{eq:fin4}
{\mathbb P}_a(I\ge (1+\varepsilon/2)k)\le \exp(-c\varepsilon^2 k).
\end{equation}

Next, we estimate the moments of the off-diagonal part $II$.
We use Lemma \ref{decoup} twice, which gives
$$
\E_{a,\varkappa} |II|^p\le 16^p \E_{a,a',\varkappa,\varkappa'} |II'|^p:=
16^p \E_{a,a',\varkappa,\varkappa'}
\biggl|\sum_{j=0}^{k-1}\sum_{i\not=i'} a_i a'_{i'}\varkappa_{j+i}\varkappa'_{j+i'}x_{j+i}x_{j+i'}\biggr|^p,
$$
where $a'$ and $\varkappa'$ are independent copies of $a$ and $\varkappa$, respectively.

First, we make a substitution $v=j+i, v'=j+i'$ and use the Khintchine inequality 
with the optimal constant $c_p\le \sqrt p$ and the random variable $\varkappa$ to obtain
\begin{align*}
{\mathbb E}_{\varkappa}
\biggl|\sum_{j=0}^{k-1}\sum_{i\not=i'}
a_i a'_{i'}\varkappa_{j+i}\varkappa'_{j+i'}x_{j+i}x_{j+i'}\biggr|^p&=
{\mathbb E}_{\varkappa}
\biggl|\sum_{v=0}^{d-1} \varkappa_v x_v
\sum_{v'\not=v} \varkappa'_{v'}x_{v'}
\sum_{j=0}^{k-1} a_{v-j} a'_{v-j'}\biggr|^p\\
&\le c_p^p \biggl(\sum_{v=0}^{d-1}x_v^2 \Bigl(\sum_{v'\not=v} \varkappa'_{v'}x_{v'}
\sum_{j=0}^{k-1} a_{v-j} a'_{v-j'}\Bigr)^2\biggr)^{p/2}.
\end{align*}

Next, we involve Minkowski's inequality with respect to $p/2\ge 1$ and Khintchine's
inequality for the random variable $\varkappa'.$
\begin{align*}
{\mathbb E}_{\varkappa, \varkappa'}|II'|^p
&\le c_p^p\, {\mathbb E}_{\varkappa'}\biggl(\sum_{v=0}^{d-1}x_v^2 \Bigl(\sum_{v'\not=v} \varkappa_{v'}x_{v'}
\sum_{j=0}^{k-1} a_{v-j} a'_{v-j'}\Bigr)^2\biggr)^{p/2}\\
&\le c_p^p \biggl(\sum_{v=0}^{d-1}x_v^2\biggl(
{\mathbb E}_{\varkappa'}\Bigl|\sum_{v'\not=v}\varkappa_{v'}x_{v'}
\sum_{j=0}^{k-1} a_{v-j} a'_{v-j'}\Bigr|^p\biggr)^{2/p}\biggr)^{p/2}\\
&\le c_p^{2p}\biggl(\sum_{v\not=v'}x_v^2x_{v'}^2
\Bigl(\sum_{j=0}^{k-1}a_{v-j}a'_{v'-j}\Bigr)^2\biggr)^{p/2}.
\end{align*}

Furthermore, the Minkowski inequality for $a$ and $a'$ gives
$$
{\mathbb E_{a,a',\varkappa,\varkappa'}}|II'|^p\le c_p^{2p}
\biggl(\sum_{v\not=v'}x_v^2x_{v'}^2\Bigl({\mathbb E}_{a,a'}
\Bigl|\sum_{j=0}^{k-1}a_{v-j}a'_{v'-j}\Bigr|^p\Bigr)^{2/p}\biggr)^{p/2}.
$$
If $a_0,\dots,a_{d-1}$ are Bernoulli variables, then Khintchine's inequality gives
$$
\biggl({\mathbb E}_{a,a'}\Bigl|\sum_{j=0}^{k-1}a_{v-j}a'_{v'-j}\Bigr|^p\biggr)^{1/p}\le \sqrt{kp},
$$
as the product of two independent Bernoulli variables
is again of this type.

For normal variables, we use first Khintchine's inequality and spherical coordinates
to obtain
\begin{align}\notag
{\mathbb E}_{a,a'}\Bigl|\sum_{j=0}^{k-1}a_{v-j}a'_{v'-j}\Bigr|^p&={\mathbb E}_{a,a'}\Bigl|\sum_{j=0}^{k-1}a_ja'_j\Bigr|^p
\le c_p^p {\mathbb E}_{a}\Bigl(\sum_{j=0}^{k-1}|a_j|^2\Bigr)^{p/2}\\
\label{eq:chaos}&=c_p^p{\mathbb E}_{a}||a||_2^p
=\frac{c_p^p}{(2\pi)^{k/2}}\int_{\R^k}e^{-||a||^2_2/2}||a||_2^pda\\
\notag &=\frac{c_p^p}{(2\pi)^{k/2}}\cdot A_k\cdot \int_0^\infty e^{-r^2/2}r^{p+k-1}dr,
\end{align}
where
$$
A_k=\frac{2\pi^{k/2}}{\Gamma(k/2)}
$$
is the area of the unit ball in $\R^k$.

We combine \eqref{eq:chaos} with Stirling's inequality and obtain
$$
\biggl({\mathbb E}_{a,a'}\Bigl|\sum_{j=0}^{k-1}a_{v-j}a'_{v'-j}\Bigr|^p\biggr)^{1/p}
\le \sqrt 2 c_p \biggl[\frac{\Gamma((k+p)/2)}{\Gamma(k/2)}\biggr]^{1/p}\le c\sqrt{p(k+p)}.
$$

Hence, if $a_0,\dots,a_{d-1}$ are independent Bernoulli or normally distributed variables, we may estimate
\begin{equation}\label{eq:chaos2}
\left({\mathbb E_{a,a',\varkappa,\varkappa'}}|II'|^p\right)^{1/p}\le c p \cdot \sqrt{(k+p)p}\cdot||x||^2
=c p^{3/2}\sqrt {k+p}.
\end{equation}
Markov's inequality then gives
$$
{\mathbb P}_{a,a',\varkappa,\varkappa'}(|II'|>k\varepsilon/2)=
{\mathbb P}_{a,a',\varkappa,\varkappa'}\biggl(\frac{2^p|II'|^p}{k^p\varepsilon^p}\ge 1\biggr)
\le\frac{2^p{\mathbb E}_{a,a',\varkappa,\varkappa'}|II'|^p}{k^p\varepsilon^p}
\le \left(\frac{cp^{3/2}\sqrt{k+p}}{k\varepsilon}\right)^{p}.
$$
We choose $p$ by the condition $\frac{\sqrt {2}cp^{3/2}}{\sqrt k\varepsilon}=e^{-1}$. 
We may assume $c\ge 1$, which ensures that $p\le k$ and 
$\frac{\sqrt{k+p}}{k}\le \frac{\sqrt 2}{\sqrt k}$, which leads to
\begin{equation}\label{eq:fin5}
{\mathbb P}_{a,a',\varkappa,\varkappa'}(|II'|>k\varepsilon/2)\le \exp(-c'(k\varepsilon^2)^{1/3}).
\end{equation}
The proof then follows by \eqref{eq:fin1} and \eqref{eq:fin2}
combined with \eqref{eq:fin3}, \eqref{eq:fin4} and \eqref{eq:fin5}.
\end{proof}

The proof of Theorem \ref{thm1} follows from Lemma \ref{lem1} by the union bound over all $\binom{n}{2}$ pairs of points.

\begin{rem} (i) We note that \eqref{eq:chaos2} follows directly by very well known estimates
of moments of Gaussian chaos, cf. \cite{HW,L}. We preferred to give a simple and direct proof.

(ii) Let us also mention that Lemma \ref{lem1} fails, if the multiplication with $D_\varkappa$
is omitted. Namely, let $k\le d$ be natural numbers,
let $a_0,\dots,a_{d-1}$ be independent normal variables and let $x=\frac{1}{\sqrt d} (1,\dots,1)$.
If $f(x)=M_{a,k}x$, then
$$
||f(x)||_2^2=k\Bigl(\sum_{j=0}^{d-1}\frac{a_j}{\sqrt d}\Bigr)^2.
$$
Due to the 2-stability of the normal distribution, the variable
$$
b:=\sum_{j=0}^{d-1}\frac{a_j}{\sqrt d}
$$
is again normally distributed, i.e. $b\approx N(0,1).$ Hence
$$
{\mathbb P}_a\Bigl(||f(x)||_2^2>(1+\varepsilon)k\Bigr)=
{\mathbb P}_b\Bigl(b^2>(1+\varepsilon)\Bigr)
$$
depends neither on $k$ nor on $d$ and Lemma \ref{lem1} cannot hold.

(iii) The statement of Theorem \ref{thm1} holds also for matrices with Toeplitz structure. The proof
is literally the same, only notational changes are necessary.
\end{rem}

{\bf Acknowledgement:} 
We thank Albrecht B\"ottcher for his comments to the topic.
The research of Aicke Hinrichs was supported by the DFG Heisenberg grant HI 584/3-2.
Jan Vyb\'\i ral acknowledges the financial
support provided by the FWF project Y 432-N15 START-Preis “Sparse Approximation
and Optimization in High Dimensions”.


\thebibliography{99}
\bibitem{A} D.~Achlioptas, Database-friendly random projections: Johnson-Lindenstrauss with binary coins.
\emph{J. Comput. Syst. Sci.}, 66(4):671-687, 2003.

\bibitem{AC2} N.~Ailon and B.~Chazelle, Approximate nearest neighbors and the fast Johnson-Lindenstrauss transform.
In \emph{Proc. 38th Annual ACM Symposium on Theory of Computing}, 2006.

\bibitem{AC} N.~Ailon and B.~Chazelle, The fast Johnson-Lindenstrauss transform
and approximate nearest neighbors. \emph{SIAM J. Comput.} 39 (1), 302-322, 2009.

\bibitem{B1} W. Bajwa, J. Haupt, G. Raz, S. Wright and R. Nowak, Toeplitz-structured
compressed sensing matrices. \emph{IEEE Workshop SSP}, 2007.

\bibitem{B2} W. U. Bajwa, J. Haupt, G. Raz and R. Nowak, Compressed channel
sensing. In \emph{Proc. CISS’08}, Princeton, 2008.

\bibitem{BT} J. Bourgain and L. Tzafriri, Invertibility of ’large’ submatrices with
applications to the geometry of Banach spaces and harmonic analysis.
\emph{Israel J. Math.}, 57(2):137-224, 1987.

\bibitem{DG} S.~Dasgupta and A.~Gupta, An elementary proof of a theorem of Johnson and
Lindenstrauss. \emph{Random. Struct. Algorithms}, 22:60-65, 2003.


\bibitem{HW} D. L. Hanson and F. T. Wright, A bound on tail probabilities for quadratic forms
in independent random variables, \emph{Ann. Math. Statist.} 42:1079-1083, 1971.

\bibitem{IM} P.~Indyk and R.~Motwani, Approximate nearest neighbors: Towards removing the curse of
dimensionality. In \emph{Proc. 30th Annual ACM Symposium on Theory of Computing}, pp. 604-613, 1998.

\bibitem{IN} P.~Indyk and A.~Naor, Nearest neighbor preserving embeddings.
\emph{ACM Trans. Algorithms}, 3(3), Article no. 31, 2007.

\bibitem{JL} W.~B.~Johnson and J.~Lindenstrauss, Extensions of Lipschitz mappings into a Hilbert space.
\emph{Contem. Math.}, 26:189-206, 1984.

\bibitem{L} R. Lata\l a, Estimates of moments and tails of Gaussian chaoses,
\emph{Ann. Prob.} 34(6):2315-2331, 2006.

\bibitem{LM} B.~Laurent and P. Massart, Adaptive estimation of a quadratic functional 
by model selection.  \emph{Ann. Statist.}  28(5):1302--1338, 2000.

\bibitem{M} J.~Matou\v{s}ek, On variants of the Johnson-Lindenstrauss lemma,
\emph{Random Struct. Algorithms}  33(2):142--156, 2008.

\bibitem{R} H.~Rauhut, Circulant and Toeplitz matrices in compressed sensing, 
In \emph{Proc. SPARS'09}, Saint-Malo, France, 2009.
\end{document}